\documentclass[11]{amsart}
\usepackage{mathrsfs}
\usepackage{amssymb,amsfonts}
\usepackage{amssymb, graphicx}
\usepackage{graphicx}
\numberwithin{equation}{section}
\usepackage{topcapt}
\usepackage[all]{xy}
\usepackage{color}
\usepackage{amsfonts}
\usepackage{amsmath}
\usepackage{latexsym}
\usepackage{amscd}
\usepackage{verbatim}
\usepackage{enumerate}
\usepackage{faktor}

\allowdisplaybreaks[1]

\newtheorem{thm}{Theorem}[section]

\newtheorem{lem}{Lemma}[section]

\newtheorem{rem}[thm]{Remark}

\title{Cohomology of the flag variety under PBW degenerations}

\author{Martina Lanini}
\email{lanini@mat.uniroma2.it}

\author{Elisabetta Strickland}
\email{strickla@mat.uniroma2.it}

\address{Dipartimento di Matematica, Universit\`a di Roma ``Tor Vergata'',  Via della Ricerca Scientifica 1, I-00133 Rome, Italy}

\begin{document}
\begin{abstract}
PBW degenerations are a particularly nice family of proper flat degenerations of type {\tt A} flag varieties. We show that the cohomology of any PBW degeneration of the flag variety surjects onto the cohomology of the original flag variety, and that this holds in an equivariant setting too. We also prove that the same is true in the symplectic setting when considering Feigin's linear degeneration of the symplectic flag variety.
\end{abstract}

\maketitle

\section{Introduction}
Degenerate flag varieties were introduced by Feigin in \cite{Fei} by Lie theoretic methods, and extensively investigated afterwards, from several viewpoints. In type ${\tt A}$ they admit a linear algebraic description, which inspired a series of papers \cite{CFR12}, \cite{CFR13}, by Cerulli Irelli, Feigin and Reineke, where they produce a realisation of these degenerations  in terms of quiver Grassmannians which was exploited, for example,  to produce a cellularization or to study the singular locus. In \cite{CL} and \cite{CLL} it was shown that in types ${\tt A}$ and ${\tt C}$, Feigin's degenerations of flag varieties are isomorphic to Schubert varieties in an appropriate partial flag variety, hence explaining many of their good properties which had been already noticed, such as, for instance,  their normality and Cohen-Macaulyness.

Recently, more general linear degenerations of type ${\tt A}$ flag varieties have been studied in \cite{CFFFR}. In particular, a large class of proper flat degenerations of flag varieties, called \textbf{PBW degenerations}, turn out to be isomorphic to Schubert varieties.

In the present paper, we show that all PBW degenerations of the flag variety $\mathcal{F}l_n$ have a further good property: their cohomology (with $\mathbb{Z}$-coefficients) surjects onto the cohomology of $\mathcal{F}l_n$. Moreover, both varieties are equipped with an action of an $(n-1)$-dimensional complex algebraic torus, and the same surjectivity result holds for the equivariant cohomology groups with integer coefficients.

The way to compare the cohomology of a complex algebraic variety $X$ and that of a proper flat degeneration $Y$ of $X$ goes as follows. One  considers a proper flat family $\pi:\widetilde{X}\to\mathbb{C}$
with  $X=\pi^{-1}(1)$, $Y=\pi^{-1}(0)$. Then  one knows (see \cite {C} or  \cite {P}) that since   $\widetilde{X}$ is proper and flat, it  contracts to $Y$ and so there is a map
$$g:H^*(Y,\mathbb{Z})\cong H^*(\widetilde{X},\mathbb{Z})\rightarrow H^*(X,\mathbb{Z}),$$
induced by the inclusion of $X$ in $\widetilde{X}$. 

We want to conclude this introduction by recalling that in general the surjectivity of $g$ is not to be expected, and the fact that it holds for \textbf{PBW degenerations} shows once more that these degenerations are extremely well-behaved.
 A particularly nice and easy example of this failure is the toric degeneration of the Grassmannian of 2-planes in $\mathbb{C}^4$ (identified under the Pl\"ucker embedding with the Klein quadric inside $\mathbb{P}^5$) given by
$$\widetilde{X}_t =\{[Z_{12}, Z_{13}, Z_{14}, Z_{23}, Z_{24}, Z_{34}] \in\mathbb{P}^5 \mid Z_{12}Z_{34} - Z_{13}Z_{24} + tZ_{14}Z_{23} =0\},$$
in this case the homomorphism $g$ is neither injective nor surjective (cf. \cite[Proposition 5.1(3)]{IX}). This small example should not induce the reader to believe that surjectivity fails always for toric degenerations: by the main result of this paper --or by direct computation-- one can see that surjectivity holds for Feigin's linear degeneration of $\mathcal{F}l_3$, which in this case is toric (and coincides with the Gelfan'd-Tsetlin degeneration of $\mathcal{F}l_3$).

\section{Flag varieties, Schubert varieties and their cohomology}

 In this section we collect some classical results about type $\tt A$ flag varieties and their Schubert varieties (see, for example, \cite{Fu}).

Let $\underline{d}=(d_1<d_2<\ldots<d_r)$ be a sequence of strictly increasing positive integers and let $n> r$. Let $V$ be an $n$-dimensional complex vector space. We denote by $\mathcal{F}l_{\underline{d},n}$ the variety of (partial) flags:
$$\mathcal{F}l_{\underline{d},n}=\{U_{d_1}\subset U_{d_2}\subset \ldots \subset U_{d_{r}}\mid U_k\in Gr(k,V)\}.$$

 If $d=(1,2,\ldots, n-1)$, we write $\mathcal{F}l_n$ instead of $\mathcal{F}l_{\underline{d},n}$.

The action of $SL_n$ on $V$ induces a transitive action  on $\mathcal{F}l_{\underline{d},n}$. Fix an ordered basis $(e_1, e_2, \ldots, e_n)$ of $V$ and for any $i=1, \ldots, n$ denote by $E_i=\textrm{span}_\mathbb{C}\{e_1, \ldots, e_i\}$. Then $\mathcal{F}l_{\underline{d},n}$ is the $SL_n$-orbit of the flag $E_\bullet=(E_{d_1}\subseteq E_{d_2}\subseteq\ldots\subseteq E_{d_r})$. Given a permutation $w\in \mathfrak{S}_n$, we write $E^w_\bullet$ for the coordinate flag whose $i$-th space is
$$E^w_{d_i}=\textrm{span}_{\mathbb{C}}\{e_{w(1)}, e_{w(2)}, \ldots, e_{w(d_i)}\}.$$
Let $W_{\underline{d}}$ be the stabiliser of $(d_1, \ldots, d_r)$ in  $\mathfrak{S}_n$ and denote by $\mathfrak{S}_n^{\underline{d}}$ the set of minimal length coset representatives in $\mathfrak{S}_n/W_{\underline{d}}$. Let $B\subset SL_n$ be the Borel subgroup of upper triangular matrices.  Then $\mathcal{F}l_{\underline{d},n}$ is a CW-complex with cells  $BE^w_\bullet$ for $w\in \mathfrak{S}_n^{\underline{d}}$, all of even real dimension. Schubert varieties are the closures $X_w:=\overline{BE^w_\bullet}$, and their  fundamental homology classes constitute a $\mathbb{Z}$-basis of $H_*(\mathcal{F}l_{\underline{d}, n},\mathbb{Z})$. A $\mathbb{Z}$-basis for $H^*(\mathcal{F}l_{\underline{d},n})$ is obtained taking duals of these classes, the Schubert classes. 

Denote by $\leq$ the Bruhat order on $\mathfrak{S}_{n}$. Any Schubert variety inherits a structure of CW-complex from $\mathcal{F}l_{\underline{d},n}$, as it is a disjoint union of $B$-orbits:
$$X_w=\bigsqcup_{\substack{y\leq w\\ y\in\mathfrak{S}^m_{\underline{d}}}} BE^y_\bullet.$$
Therefore also the integral cohomology of $X_w$ is determined by the integral homology of it. This will play an important role in the proof of our main result, where instead of dealing with 2-cocycles, we will be allowed to  work with 2-cycles.

Recall that in the case of the variety of complete flags, $H_*(\mathcal{F}l_{n},\mathbb{Z})$  is a ring generated in degree 2 by the classes of the Schubert varieties $X_{s_i}$, where $s_i$ is the simple transposition which exchanges $i$ and $i+1$:
\begin{equation}\label{DefinSi}
X_{s_i}=\left\{
E_1\subset E_2 \subset \ldots  \subset E_{i-1}\subset U\subset E_{i+1}\subset \ldots \subset E_{n-1}\mid U\in Gr(i,V)\
\right\}.
\end{equation}
The fact that these varieties (of complex dimension 1) also lie in the special fibre of the PBW degenerations, which  are going to be introduced in the next  section, will be crucial for us.

In general, given a complex semisimple algebraic group $G$ with Borel $B$, the corresponding generalised flag variety $G/P$ (where $P\supseteq B$ is a parabolic subgroup) is also a CW-complex, whose cells have even real dimension. In particular $H_2(G/B,\mathbb Z)$ is free with  basis the homology classes of the Schubert varieties $X_{s_i}$, indexed by simple reflections of the Weyl group $W$ of $G$.  Moreover, $H^*(G/B, \mathbb{Z})$ is generated, as a ring, in degree 2, see \cite{Bo}.

Finally, we want to mention an alternative presentation of the cohomology of the flag variety $\mathcal{F}l_n$, due to Borel \cite{Bo}, in terms of invariant rings. 
The symmetric group $\mathfrak{S}_n$ naturally acts on the polynomial ring $S:=\mathbb{Z}[x_1,x_2, \ldots, x_n]$ by permuting the variables, and we have
\begin{equation}\label{Eqn:CohFlagsBorel}H^*(\mathcal{F}l_{\underline{d},n}, \mathbb{Z})\cong\left(({S}/{(S^{\mathfrak{S_n}}_+)})\right)^{W_{\underline{d}}},
\end{equation}
where $(S^{\mathfrak{S_n}}_+)$ denotes the ideal generated by the  $\mathfrak{S}_n$-homogeneous invariants of positive degree. As for a Schubert variety $X_w$, the inclusion $X_w\hookrightarrow \mathcal{F}l_{\underline{d},n}$ induces the presentation
\begin{equation}\label{Eqn:CohSchubertsBorel}H^*(X_w, \mathbb{Z})\cong\left(\faktor{({S}/{(S^{\mathfrak{S_n}}_+)})}{I_w}\right)^{W_{\underline{d}}},
\end{equation}
where $I_w$ is the $\mathbb{Z}$-span of the Schubert classes $[X_u]^*$, for $u\not\leq w$, $u\in\mathfrak{S}_n^{\underline{d}}$.

\section{PBW-degenerations of flag varieties}
We recall here the definition of PBW degenerations of the flag variety from \cite{CFFFR}.

Let $V$ be a complex, $n$-dimensional vector space, and let $\{e_1, \ldots, e_n\}$ be an ordered basis of $V$.  

For any $t\in \mathbb{C}$, and $i=1, \ldots n$, we define the map
$$\textrm{pr}_{i,t}(e_j)=\begin{cases} e_j\ \ \ \text{if}\ \ j\neq i\\ te_i  \ \ \ \text{if}\ \ j=i\end{cases}.$$ 

 Moreover,  we set $pr_{0,t}:=\textrm{Id}_V$.

Let $\underline{j}=(j_1, \ldots, j_r)$ be such that $1\leq j_1<\ldots< j_r\leq n-2$, for some $r\geq 1$, and define  
$$b_k:=\begin{cases}
i, \ \ \ \text{if } k=j_i,\\
0\ \ \ \text{otherwise}. 
\end{cases}
\qquad (k=1, \ldots, n-1)
$$
We consider the variety
$$\widetilde{\mathcal {F}l}^{\underline{j}}_{n}=\{(V_1,V_2,\ldots, V_{n-1},t)| V_i\in Gr(i,V);\ \textrm{pr}_{b_i,t}(V_i)\subset V_{i+1}; \ t\in \mathbb{C}{}\}.
$$

Note that there is an obvious projection:
\begin{equation}\label{EqnDegeneration}\pi:\widetilde{\mathcal {F}l}^{\underline{j}}_{n}\to \mathbb{C},\end{equation}
given by $\pi((V_1,V_2,\ldots, V_{n-1},t))=t$,
such that $\pi^{-1}(1)$ is isomorphic to the variety $\mathcal{F}l_n$ of complete flags in $\mathbb{C}^n$.


Following \cite{CFFFR}, we denote by $\mathcal{F}l^{\underline{j}}_{n}$ the fibre over $0$, and call it a PBW degeneration.

Given our $r$-tuple $\underline{j}=(j_1, \ldots, j_r)$, we define
$$
\ell_1=1, \qquad \ell_i=\#\{z\mid 1\leq z\leq r \hbox{ and }j_z<i\}+i, \quad(i=2, \ldots, n-1).
$$

The following result generalises \cite[Theorem 1.2]{CL} and tells that every PBW degeneration of a flag variety can be realised as a Schubert variety inside an appropriate partial flag variety.
\begin{thm}[{\cite[Theorem 6]{CFFFR}}]\label{ThmDegSchubert}
We have an isomorphism of projective varieties $\mathcal{F}l^{\underline{j}}_n\stackrel{\sim}{\rightarrow} X_{w_{\underline{j}}}$, where $X_{w_{\underline{j}}}$ is a Schubert variety inside $\mathcal{F}l_{\underline{\ell},n+r}$.
\end{thm}

The above fact will be a central ingredient in the proof of our main theorem.

We recall here the explicit isomorphism, since it will be needed later. 
In the proof of \cite[Theorem 6]{CFFFR}, the isomorphism is given by using the formalism of quiver Grassmannians, we will reformulate it here so that such a formalism will not be necessary.

We denote by $\{\widetilde{e_1}, \ldots, \widetilde{e}_{n+r}\}$ the standard basis of $\mathbb{C}^{n+r}$ and consider the maps $\pi_i:\textrm{span}_\mathbb{C}\{\widetilde{e_1}, \widetilde{e_2}, \ldots, \widetilde{e}_{n+\ell_i-i}\}\rightarrow \mathbb{C}^{n}$ given by
$$
\pi_i(\widetilde{e_k})=
\begin{cases}
0, \ \ \ \text{if}\ 1\leq k\leq \ell_i-i-1, \\
e_k, \ \ \ \text{if}\ \ell_i-i\leq k\leq n,\\
e_{k-n}, \ \ \  \text{if}\ n+1\leq k.
\end{cases}
$$
Then the isomorphism of the theorem is given by 
\begin{equation}\label{EqnIsoDegSchubert}
\zeta:\mathcal{F}l_n^{\underline{j}}\rightarrow \mathcal{F}l_{\underline{l},n+r}, \qquad 
(V_1, \ldots, V_{n-1})\mapsto( \pi_1^{-1}(V_1), \ldots, \pi_{n-1}^{-1}(V_{n-1})).
\end{equation}

For $i=1, \ldots, n-1$, recall the Schubert variety
\begin{equation*}
X_{s_i}=\left\{
E_1\subset E_2 \subset \ldots  \subset E_{i-1}\subset U\subset E_{i+1}\subset \ldots \subset E_{n-1}\mid U\in Gr(i,V)\
\right\}.
\end{equation*}
Denote 
$$\stackrel{\circ}{X_{s_i}}:=X_{s_i}\setminus \{E_\bullet, E_\bullet^{(i, i+1)}\}.$$

\begin{lem}\label{LemmaLinIndp}
 \begin{enumerate}
\item The variety $X_{s_i}\times\{t\}$ is contained in $\pi^{-1}(t)$ for any $t\in \mathbb{C}$.
 \item The torus $T^0\subset SL_{n+r}$ of complex  diagonal matrices acts on $\zeta(\stackrel{\circ}{X_{s_i}})$ via the character 
$$\textrm{diag}(\mu_1,\mu_2, \ldots,\mu_{n+r+1})\mapsto \mu_i\mu^{-1}_{i+1}.$$
\end{enumerate}
\end{lem}
\begin{proof}
\begin{enumerate}
\item
Clearly, $\textrm{pr}_{b_k,t}(E_k)\subseteq E_k$ for any $k=1, \ldots n-2$ and $t\in \mathbb{C}$, so that $\textrm{pr}_{b_k,t}(E_k)\subset E_{k+1}$ for all $k=1, \ldots, n-2$ and for all $t\in\mathbb{C}$,  and $\textrm{pr}_{b_{i-1},t}(E_{i-1})\subset U$. Moreover, since there exists $[a_1 : a_2]\in \mathbb{P}^1$ such that $U=E_{i-1}\oplus\mathbb{C}(a_1 e_i+a_2 e_{i+1})$, we also have $\textrm{pr}_{b_i,t}(U)\subseteq E_{i+1}$. 
\item For $j\leq k$, we denote by $\widetilde{E}_{[j,k]}$ the $\mathbb{C}$-span of the vectors $\{\widetilde{e}_j, \widetilde{e}_{j+1}, \ldots, \widetilde{e}_{k-1}, \widetilde{e}_{k}\}$, and we abbreviate $\widetilde{E}_k:=\widetilde{E}_{[1,k]}$. Moreover, we use the convention that $\widetilde{E}_{[j,k]}=\{0\}$ if $k<j$.
We have
$$\zeta(E_k)=\pi_k^{-1}(E_k)=\widetilde{E}_k\oplus\widetilde{E}_{[n+1, n+\ell_k-k]},$$
so that 
\begin{align*}\zeta(X_{s_i})&=\left\{
\zeta(E_1)\subset \ldots  \subset \zeta(E_{i-1})\subset
 \zeta(E_{i-1})\oplus \mathbb{C}(a_1\zeta(e_i)+a_2\zeta(e_{i+1}))\subset \zeta(E_{i+1})\subset \ldots \subset \zeta(E_{n-1})\mid [a_1:a_2]\in\mathbb{P}^1\
\right\}\\
&=\left\{
\zeta(E_1)\subset \ldots  \subset \zeta(E_{i-1})\subset
 \zeta(E_{i-1})\oplus \mathbb{C}(a_1\widetilde{e}_i+a_2\widetilde{e}_{i+1})\subset \zeta(E_{i+1})\subset \ldots \subset \zeta(E_{n-1})\mid [a_1:a_2]\in\mathbb{P}^1\
\right\},
\end{align*}
where the second equality follows from the fact that $i>\ell_i-i$, so that $\pi_i^{-1}(e_k)=\widetilde{e}_{k}$ for all $k>i$. At this point it is clear that $T^0$ acts on
$$
\zeta(\stackrel{\circ}{X_{s_i}})=\left\{
\zeta(E_1)\subset \ldots  \subset \zeta(E_{i-1})\subset
 \zeta(E_{i-1})\oplus \mathbb{C}(a_1\widetilde{e}_i+a_2\widetilde{e}_{i+1})\subset \zeta(E_{i+1})\subset \ldots \subset \zeta(E_{n-1})\mid a_1,a_2\in\mathbb{C}^\times\
\right\},
$$
via the character $\textrm{diag}(\mu_1,\mu_2,\ldots,\mu_{n+r})\mapsto \mu_i\mu_{i+1}^{-1}$.
\end{enumerate}
\end{proof}

\section{Main result}

Let us start by remarking that there is an isomorphism
$$\widetilde{\mathcal {F}l}^{\underline{j}}_{n}\setminus \mathcal{F}l^{\underline{j}}_{n}\rightarrow \mathcal {F}l_{n}\times \mathbb{C}^\times, \qquad 
(V_1,V_2,\ldots ,V_{n-1},t)\mapsto(V_1,\textrm{pr}_{b_1,t}^{-1}(V_2),\ldots,\textrm{pr}_{b_1,t}^{-1}\cdots \textrm{pr}_{b_{n-2},t}^{-1}(V_{n-1}),t).$$

Moreover, by \cite{CFFFR}, the degeneration \eqref{EqnDegeneration} is proper and  flat so that  $\widetilde{\mathcal {F}l}^{\underline{j}}_{n}$ is the closure of $\pi^{-1}(\mathbb C^\times)$.   This implies (see, for example, \cite{C} or \cite{P}) that  we get a contraction of $\widetilde{\mathcal {F}l}^{\underline{j}}_{n}$ onto $\mathcal {F}l^{\underline{j}}_{n}$. In particular,  $H^*(\mathcal {F}l_n^{\underline{j}},\mathbb Z)\cong H^*(\widetilde{\mathcal{F}l}_n,\mathbb Z)$.

Since $\pi^{-1}(1)=\mathcal{F}l_n\hookrightarrow \widetilde{\mathcal {F}l}^{\underline{j}}_{n}$, we obtain a homomorphism of graded rings
$$g: H^*(\mathcal {F}l_n^{\underline{j}},\mathbb Z)\to H^*(\mathcal {F}l_n,\mathbb Z).$$

Our main result is that the above homomorphism is surjective.

\begin{thm}The homomorphism $g$ is surjective.
\end{thm}
\begin{proof} Since
$H^*(\mathcal {F}l_n,\mathbb Z)$ is generated by $H^2(\mathcal {F}l_n,\mathbb Z),$ (see \cite{Bo} or \cite{Fu} pp.131-153), it is enough to show that
$$g^2: H^2(\mathcal {F}l_n^{\underline{j}},\mathbb Z)\to H^2(\mathcal {F}l_n,\mathbb Z)$$
is surjective. By Theorem \ref{ThmDegSchubert},  $H^*(\mathcal{F}l^{\underline{j}}_n, \mathbb{Z})\cong H^*(X_{w_{\underline{j}}}, \mathbb{Z})$ as graded $\mathbb{Z}$-modules,  and we deduce that $H^*(\mathcal{F}l^{\underline{j}}_n, \mathbb{Z})$ is torsion free. Hence, it is sufficient to prove that the dual homomorphism in homology 
$$g_2:H_2(\mathcal{F}l_{n},\mathbb Z)\to H_2(\mathcal{F}l^{\underline{j}}_n,\mathbb Z)$$
is injective and its image is a split direct summand. 

To prove this, we will use the known fact that $H_2(\mathcal{F}l_{n},\mathbb Z)$ is spanned  by Schubert cycles
$$
X_{s_i}:=\left\{
E_1\subset E_2 \subset \ldots  \subset E_{i-1}\subset U\subset E_{i+1}\subset \ldots \subset E_{n-1}\mid U\in Gr(i,V)\
\right\}.
$$
Recall that a basis for $H_2(\mathcal {F}l_n,\mathbb Z)$ is given by the cycles  $c_i$, $i=1,\ldots , n-1$, of the varieties $X_{s_i}$ defined above.

By Lemma \ref{LemmaLinIndp}(1), $X_{s_i}$ is contained in $\pi^{-1}(t)$ for any $t\in\mathbb{C}$. Such a containment induces a map $X_{s_i}\times [0,1]\rightarrow\widetilde{\mathcal{F}l}^{\underline{j}}_n$, which gives a homotopy in $\widetilde{\mathcal F}l^{\underline{j}}_n$ between  $X_{s_i}\times\{1\}$ and $X_{s_i}\times\{0\}$. In particular the cycles $X_{}i\times\{1\}$ and $X_{s_i}\times\{0\}$ are homologous in $\widetilde{\mathcal F}l^{\underline{j}}_n$ and  we deduce that the class of $X_{s_i}$ in $H_2(\mathcal{F}l_n^{\underline{j}},\mathbb Z)$ is the image of the class of $X_{s_i}$ in $H_2(\mathcal {F}l_n,\mathbb{Z})$. 
For what we have noticed, these are cycles of 1-dimensional (complex) subvarieties of $\mathcal{F}l_n^{\underline{j}}$ and by Lemma \ref{LemmaLinIndp}(2) they are linearly independent and part of a basis. Hence the claim follows.
\end{proof}

\begin{rem}It is natural to ask whether the above surjectivity result can  be extended  to all flat and irreducible degenerations of \cite{CFFFR}. In \cite{CFFFR}, a Bialynicki-Birula decomposition of the special fibre, say $Y$, is provided, so that also in that case one could check surjectivity by looking at the induced map between the 2-homology groups. As in the proof of our main result, it is possible to determine the image of the fundamental homology class of $X_{s_i}$ inside  $H_2(Y, \mathbb{Z})$. However, to show linear independence, the identification of the special fibre  with a Schubert varietiety in a partial flag variety of bigger rank was necessary, and we do not see at the moment an alternative argument. Such an identification is missing in the more general case, which we leave to future work.
\end{rem}

\subsection{The equivariant case}
Let us denote by $T\subset SL_n(\mathbb{C})$ the algebraic torus consisting of diagonal matrices with respect to the basis $\{e_1,\ldots, e_n\}$.

The torus $T$ acts on $\mathbb{C}^n$ by rescaling the coordinates: if $\underline{\lambda}=\textrm{diag}(\lambda_1,\ldots, \lambda_n)\in T$ and $v=\sum a_j e_j\in \mathbb{C}^n$, then $\underline{\lambda}\cdot v=\sum \lambda_j a_j e_j$.
Now, $v\in V_i$ if and only if $\underline{\lambda}\cdot v\in \underline{\lambda}V_i$, so if this holds and $\textrm{pr}_{b_i,t}V_i\subset V_{i+1}$, then
 $$pr_{b_i,t}(\underline{\lambda} v)=\sum_{j\neq b_i}\lambda_ja_je_j=\underline{\lambda} pr_{b_i,t}(v)\in \underline{\lambda}V_{i+1}.$$
Therefore, if
$(V_1,V_2,\ldots, V_{n-1},t)\in  \widetilde{\mathcal {F}l}_{n}$, then $(\underline{\lambda} V_1,\underline{\lambda} V_2,\ldots \underline{\lambda} V_{n-1},t)\in  \widetilde{\mathcal {F}l}_{n}$. The torus action preserves any fibre of the map $\pi$ and, hence, we have a homomorphism
$$H_T^*(\mathcal{F}l^{\underline{j}}_n)\rightarrow H^*_T(\mathcal{F}l_n).$$

\begin{thm}The homomorphism $H_T^*(\mathcal{F}l^{\underline{j}}_n)\rightarrow H^*_T(\mathcal{F}l_n)$ is surjective.
\end{thm}
\begin{proof}
The statement follows once noticed that the cycles $X_{s_i}$ are stabilised by the torus $T$ and hence they define equivariant cycles both in $\mathcal{F}l_n$ and in $\mathcal{F}l^{\underline{j}}_n$, which (by Lemma \ref{LemmaLinIndp}(2)) are in both linearly independent.
\end{proof}

\begin{rem}The homomorphism $g$ is not injective if $n\geq 2$. The total dimension of the cohomology of the Schubert variety $X_{w_{\underline{j}}}$ coincides with the number of 
its $T^0$-fixed points ($T^0$ being the maximal torus of diagonal matrices in $SL_{n+r}(\mathbb{C})$ as in Lemma \ref{LemmaLinIndp}), that are the coordinate
(partial) flags $\tilde{E}^y_\bullet$ in $\mathcal{F}l_{\underline{\ell}, n+r}$, for $y\leq w_{\underline{j}}$, $y\in\mathfrak{S}_{n+r}^{\underline{\ell}}$. Now we notice that any coordinate flag in $\mathcal{F}l_n$ is also contained in $\mathcal{F}l_n^{\underline{j}}$ and
that its image under $\zeta$ is a coordinate flag in $\mathcal{F}l_{\underline{\ell}, n+r}$, hence a $T^0$-fixed point. If $n>2$, then the cardinality of the set of $T^0$-fixed points 
is strictly greater than the number of coordinate flags in $\mathbb{C}^n$, which is the total dimension of the cohomology of $\mathcal{F}l_n$.
\end{rem}

\begin{rem}By the previous remark, we know that $\ker(g)\neq\{0\}$ and it would be very interesting to give an explicit description of it. Given that both the cohomology of the flag variety $H^*(\mathcal{F}l_n, \mathbb{Z})$ and the cohomology of the Schubert variety $H^*(X_{w_{\underline{j}}}, \mathbb{Z})$ admit a nice presentation (cf. Equations  \eqref{Eqn:CohFlagsBorel},  \eqref{Eqn:CohSchubertsBorel}) involving Schubert classes, one might hope to be able to describe the kernel in terms of Schubert classes. Unluckily, it does not seem to be feasible, since the embedding $\zeta$ from \eqref{EqnIsoDegSchubert} does not map in general a Schubert variety of $\mathcal{F}l_n$ to a Schubert variety inside $\mathcal{F}l_{\underline{j},n+r}$. A first example of this phenomenon can be already observed in the  the case $n=3$, $\underline{j}=\{1\}$.
\end{rem}

\section{The symplectic case}
We extend here our result to the case of Feigin's degenerations of symplectic flag varieties.

 Let $W$ be a $2n$-dimensional complex vector space. We keep the same notation as in the previous section and denote by $(\widetilde{e}_1, \widetilde{e}_2, \ldots, \widetilde{e}_{2n})$ an ordered basis for $W$. Moreover,
 we equip
$W$ with the symplectic form given by the following matrix:
$$
\left(
\begin{array}{cc}
 0&J\\
 -J&0
\end{array}
\right),
$$
where $J$ denotes the $n\times n$-antidiagonal matrix with entries $(1,1,\ldots,1)$. Given a subspace $U\subseteq W$, we denote by
$U^\perp$ its orthogonal space in $W$ with respect to the above symplectic form. 
 Let $\underline{d}=(1\leq d_1<d_2<...<d_r\leq 2n-1)$ be such that $d_{r-i+1}=2n-d_i$ for any $i$.
 Then one can define an involution 
$$\iota:\mathcal{F}l_{\underline{d},2n}\rightarrow \mathcal{F}l_{\underline{d},2n}\qquad (W_{d_i})\mapsto (W_{d_i}')$$
with $W_{d_i}'=W_{2n-d_i}^\perp$. The symplectic flag variety 
$Sp\mathcal{F}l_{\underline{d},2n}$ can hence be realised as the subvariety of flags in $\mathcal{F}_{\underline{d},2n}$ which are fixed by $\iota$.

On the other hand, also Feigin's degeneration of the symplectic flag variety can be obtained by taking fixed points of an involutive automorphism of the  type ${\tt A}$ degeneration, as proven 
in \cite{FFL}. 

Let $V$ be a $2n$-dimensional complex vector space, with basis $\{e_1,\cdots, e_{2n}\}$. As in \cite[\S 4.41]{CL}, we equip the vector space $V$ with a non-degenerate skew-symmetric bilinear 
form $b_V[\cdot, \cdot]$ such that
\begin{equation}\label{Eq:DefBilFormV}
e_{k}^\ast= \left\{\begin{array}{rcl}e_{2n-1-k} &\textrm{if} & 1\leq k\leq 2n-2,\\ e_{2n}&\textrm{if} & k=2n-1.\end{array}\right.\end{equation}
Again, for a subspace $Z\subseteq V$, we write $Z^\perp$ for its orthogonal space in $V$ with respect to the form $b_V[\cdot, \cdot]$.
 
Thus, one can consider inside
$\widetilde{\mathcal{F}l}^{(1,2,\ldots, 2n-2)}_{2n}$ the subvariety of isotropic elements, that is 
$$
Sp\widetilde{\mathcal{F}l}^{(1,2,\ldots, 2n-2)}_{2n}:=\{(V_1,V_2,\ldots, V_{2n-1},t)| V_i\in Gr(i,V);\ \textrm{pr}_{b_i,t}(V_i)\subset V_{i+1}; \ V_{2n-i}=V_{i}^\perp; \ t\in\mathbb{C}\}.
$$
Again, we consider the projection $\pi:Sp\widetilde{\mathcal{F}l}^{(1,2,\ldots,2n-2)}_{2n}\rightarrow \mathbb{C}$ given by $(V_1, \ldots, V_{2n-1},t)\mapsto t$. The fibre over $t\neq 0$ is isomorphic to the symplectic flag variety
$$Sp\mathcal{F}l_{2n}\cong\{(V_1\subset V_2\subset \ldots \subset V_{2n-1})| V_i\in Gr(i,V); \ V_{2n-i}=V_{i}^\perp\}$$
and we denote by $Sp\mathcal{F}l^a_{2n}$ the fibre over $0$, following the notation in \cite{FFL, CL}.

Any fibre of the homomorphism $\pi:\widetilde{\mathcal{F}l}^{(1,2,\ldots, 2n-2)}_{2n}\rightarrow \mathbb{C}$
is hence equipped with an involutive automorphism:
$$\iota^t:\pi^{-1}(t)\rightarrow \pi^{-1}(t), \qquad (V_1, V_2, \ldots, V_{2n-1}, t)\mapsto (V_{2n-1}^\perp, V_{2n-2}^\perp, \ldots, V_{1}^\perp,t).$$ 

Let $T^0\subset SL_{4n-2}$ be the maximal torus of diagonal matrices. In \cite[\S 4.1]{CL} it is proven that the following diagram of $T^0$-varieties commutes:
\begin{equation}\label{Eq:CommDiagMain}
\xymatrix{
\mathcal{F}l_{2n}^a\ar^{\iota^0}[r]\ar_\zeta[d]&\mathcal{F}l_{2n}^a\ar^\zeta[d]\\
X_{w_{(1,2,\ldots, 2n-2)}}\ar^\iota[r]&X_{w_{(1,2,\ldots, 2n-2)}}
}
\end{equation}
for an appropriate Schubert variety $X_{w_{(1,2,\ldots, 2n-2)}}$ inside $\mathcal{F}l_{(1,3,\ldots, 4n-5,4n-3), 4n-2}$

It follows that also in the symplectic case Feigin's degeneration can be realised as a Schubert variety (\cite[Theorem 4.1]{CL}), since any Schubert variety inside $Sp\mathcal{F}l_{(1,3,\ldots, 4n-5,4n-3), 4n-2}$ 
is obtained as fixed point set  of the involution $\iota$ restricted to a Schubert variety in $\mathcal{F}l_{(1,3,\ldots, 4n-5, 4n-3),4n-2}$ (cf.\cite[Proposition 6.1.1.2]{LR}).

\subsection{Symplectic version of the main result}
By \cite[Proposition 4.10]{FFL}, the map $\pi:Sp\widetilde{\mathcal{F}l}^{(1,2,\ldots, 2n-2)}_{2n}\rightarrow \mathbb{C}$ is proper and flat and we get once again a
homomorphism $g:H^\ast(Sp\mathcal{F}l_{2n}^a)\rightarrow H^\ast(Sp\mathcal{F}l_{2n})$.

\begin{thm}The homomorphism $g:H^\ast(Sp\mathcal{F}l_{2n}^a)\rightarrow H^\ast(Sp\mathcal{F}l_{2n})$ is surjective.
\end{thm}
\begin{proof}First of all, an element invariant under the involution $\iota^t$ is uniquely determined by the first $n$-vector spaces $(V_1, V_2, \ldots V_n)$, so that in this proof we will write $(V_1, V_2, V_n, t)$ for $(V_1, V_2, \ldots V_n, V^\perp_{n-1}, \ldots, V_{1}^\perp,t)$.

For $k=1, \ldots,  2n$, denote by $E_k$ the $\mathbb{C}$-span of the vectors $\{e_1, e_2, \ldots, e_k\}$. We let $F_\bullet=(F_k)$ be the symplectic flag given by $F_k=E_k$ for $k=1, \ldots n-1$, and $F_n=E_{n-1}\oplus\mathbb{C}e_{2n-1}$. Moreover, for $i=1, \ldots, n-1$, we define
$$
F_\bullet^{(i,i+1)}=(F^{(i,i+1)_k}), \qquad 
F^{(i,i+1)}_k=
\begin{cases}
F_k&\textrm{ if }k\neq i,\\
F_{i-1}\oplus\mathbb{C}e_{i+1}&\textrm{ if }k=i.
\end{cases}
$$ 
We also set
$$
F^{(n, n+1)}_\bullet=(F_k^{(n, n+1)}), \qquad 
F^{(n, n+1)}_k=
\begin{cases}
F_k&\textrm{ if }k\neq n,\\
F_{n-1}\oplus\mathbb{C}e_{n+1}&\textrm{ if }k=n.
\end{cases}
$$

Next, for $i=1, \ldots, n$, consider the varieties
$$
X^{s_i}:=\{F_1\subset F_2\subset\ldots\subset F_{i-1}\subset U\subset F_{i+1}\subset \ldots\subset F_{n}\}.
$$

At this point the proof goes exactly as in the type {\tt A} case, with the difference that the 2-cycles to be considered now are the cycles $c_i$ corresponding to the above
$X_{s_i}$.

Also in the symplectic case, the cohomology of $Sp\mathcal{F}l_{2n}$ is generated in degree 2 and hence it is sufficient to prove that the dual map restricted to the degree 2 part
$$g_2^*:H_2(Sp\mathcal{F}l_{2n})\rightarrow H_2(Sp\mathcal{F}l^a_{2n})$$
 is injective.

 The same argument as in the proof of Lemma \ref{LemmaLinIndp}(1) shows that $X_{s_i}$ is contained in every fibre $\pi^{-1}(t)$, so again we deduce that the class of $X_{s_i}$ in $H^*(Sp\mathcal{F}l^a_{2n})$ is the image under $g_2^*$  of the class of $X_{s_i}$ in $H^*(Sp\mathcal{F}l_{2n})$.

Finally, let us denote by $X^\iota_{w_{(1,2,\ldots, 2n-2)}}$ the Schubert variety of $Sp\mathcal{F}l_{(1,3,\ldots, 4n-3), 4n-2}$ which is obtained as the $\iota$-fixed points of the Schubert variety $X_{w_{(1,2,\ldots, 2n-2)}}$ of $\widetilde{\mathcal{F}l}_{(1,2,\ldots, 4n-3), 4n-2}$. The same argument as for Lemma \ref{LemmaLinIndp}(2) shows that the classes of $\zeta(X_{s_i})$ in $H^*(X^\iota_{w_{(1,2,\ldots, 2n-2)}})$ are linearly independent, and so must be  the classes of $X_{s_i}$ in $H_2(Sp\mathcal{F}l^a_{2n})$.  
\end{proof}

\end{document}